\documentclass[11pt]{amsart}
\usepackage[latin1]{inputenc}
\usepackage[english]{babel}
\usepackage{amsmath,amssymb,amsthm,amsfonts}
\usepackage{mathtools}
 \usepackage{verbatim}
\usepackage{xcolor}
\usepackage{enumerate}

\usepackage{latexsym}
\usepackage{enumerate}
\usepackage{graphicx}
\usepackage{latexsym}

\makeatletter
\@namedef{subjclassname@2010}{ \textup{2010} Mathematics Subject Classification}
\makeatother
\theoremstyle{definition}
\newtheorem{theorem}{Theorem}
\theoremstyle{definition}
\newtheorem{lemma}[theorem]{Lemma}
\theoremstyle{definition}
\newtheorem{corollary}[theorem]{Corollary}
\theoremstyle{definition}

\theoremstyle{definition}
\newtheorem{fact}[theorem]{Fact}
\theoremstyle{definition}

\theoremstyle{definition}
\newtheorem{example}[theorem]{Example}
\numberwithin{equation}{section}

\numberwithin{equation}{section}

\DeclareMathOperator{\supp}{supp}

\DeclareMathOperator*{\esssup}{ess\,sup}

\newcommand{\norm}[1]{\left\lVert#1\right\rVert}
\newcommand{\abs}[1]{\left\lvert#1\right\rvert}

\newcommand{\I}[2]{I_{#2}\left(#1\right)}

\begin{document}
\title{Pointwise multipliers of Musielak--Orlicz  spaces and factorization}
%\thanks{{\rm *}This publication has been produced during scholarship 
%period of the first author at the Lule{\aa} University of Technology, thanks to a Swedish Institute scholarschip 
%(number 0095/2013).}

%\author[Le\'snik]{Karol Le\'snik}
%\address[Karol Le{\'s}nik]{Institute of Mathematics\\
%ozna\'n University of Technology, ul. Piotrowo 3a, 60-965 Pozna{\'n}, Poland}
%\email{\texttt{klesnik@vp.pl}}
%\author[Tomaszewski]{Jakub Tomaszewski}
%\address[Jakub Tomaszewski]{Institute of Mathematics\\  Pozna\'n University of Technology, ul. Piotrowo 3a, 60-965 Pozna{\'n}, Poland}
\author{Karol Le\'snik \and Jakub Tomaszewski}
\address{Institute of Mathematics, Pozna\'n University of Technology, ul. Piotrowo 3a, 60-965 Pozna{\'n}, Poland}
\email{\texttt{klesnik@vp.pl}}
\email{\texttt{jakub.tomaszewski@put.poznan.pl}}
\maketitle

\vspace{-9mm}

\begin{abstract}
We  prove that the space of pointwise multipliers between two distinct Musielak--Orlicz spaces  is another Musielak--Orlicz space and the function defining it is given by an appropriately generalized Legendre transform. In particular, we obtain characterization of pointwise multipliers between Nakano spaces.
We also discuss factorization problem for Musielak--Orlicz spaces and exhibit some differences between Orlicz and Musielak--Orlicz cases. 
\end{abstract}

\renewcommand{\thefootnote}{\fnsymbol{footnote}}


%\maketitle
%%%%%%%%%%%%%%%%%%%%%%%%%%%
% abstract, keywords and Subject classification are optional.
%%%%%%%%%%%%%%%%%%%%%%%%%%%
%\begin{abstract}    This is a sample file.    You can use it as a guide for your submission.\end{abstract}

% Most people don't use these, so they are "commented out"
% by starting the lines with a "%"
%\begin{keywords}
%   \LaTeX, typesetting
%\end{keywords}

%\begin{AMS}
%   50C60, 18C25
%\end{AMS}

%%%%%%%%%%%%%%%%%%%%%%
% % Here is the start of the Text
%%%%%%%%%%%%%%%%%%%%%%
\section{Introduction}

Given two function spaces $X$ and $Y$ (over the same measure space), the space of pointwise multipliers $M(X,Y)$  is the space of all functions $f$ such that $fg\in Y$ for each $g\in X$. $M(X,Y)$ may be regarded as a generalized K\"othe dual space (cf. \cite{MP89,CDS08}) and a basic question is to identify $M(X,Y)$ for a given spaces $X$ and $Y$. 
%In the recent paper \cite{LT17} the authors gave a  characterization of the space of pointwise multiplier between two Orlicz function spaces. 
Many authors have investigated this problem for Orlicz spaces and many characterizations (mainly partial) have been given -- see for example Shragin \cite{Sh57}, Ando \cite{An60}, O'Neil \cite{ON65}, Zabreiko--Rutickii \cite{ZR67}, Maurey \cite{Ma74},  Maligranda--Persson \cite{MP89} and Maligranda--Nakai \cite{MN10}. In 2000 Djakov and Ramanujan settled the problem for Orlicz sequence spaces and, recently, in \cite{LT17} the authors established an analogous characterization for Orlicz function spaces. In both cases, the  space of pointwise multipliers $M(L^{\varphi_1},L^{\varphi})$ between Orlicz spaces is proved to be just another Orlicz space, i.e. 
\begin{equation}\label{intrmul}
M(L^{\varphi_1},L^{\varphi})=L^{\varphi\ominus \varphi_1},
\end{equation}
where the function $\varphi\ominus \varphi_1$ is generalized Young conjugate (generalized Legendre transform) of $\varphi_1$ with respect to $\varphi$. Observe that the above characterization generalizes, in the evident way, the classical K\H othe duality formula for Orlicz spaces, this is
\begin{equation}\label{intrmulccc}
(L^{\varphi_1})'=M(L^{\varphi_1},L^{1})=L^{\varphi_1^*},
\end{equation}
where $\varphi_1^*$ is  the Young conjugate of $\varphi_1$ (i.e. $\varphi_1^*=id\ominus \varphi_1$). 
Let us also mention here, that the identification as in (\ref{intrmul}) seems to be the most desirable, since the function  $\varphi\ominus \varphi_1$ is given in an explicit and constructive way, in contrast to  theorems from \cite{MN10} and  \cite{KLM13}, which have rather existential character (cf. \cite{Sh57,An60,ON65,ZR67,Ma74}). 

In the paper we focus on  the multipliers of Musielak--Orlicz spaces. Such investigations have been already initiated by Nakai \cite{Na16} (cf. \cite{Na17}). Under a number of  assumptions on functions $\varphi, \varphi_1$ he generalized results of \cite{MN10} to the  Musielak--Orlicz setting. Since this method is not constructive (see discussion in \cite{LT17}), we are not going to employ it. Instead of that we will use ideas of \cite{DR00} and \cite{LT17} to prove that  the representation  (\ref{intrmul}) holds also in the Musielak--Orlicz case, for an arbitrary $\sigma$ - finite measure space and without any additional assumptions on Musielak--Orlicz functions $\varphi, \varphi_1$.

The paper is organized as follows. In Section 2 we give necessary definitions on Banach ideal space and Musielak--Orlicz spaces. 
We also define the function  $\varphi\ominus \varphi_1$ (Young conjugate of $\varphi_1$ with respect to $\varphi$) for Musieak--Orlicz functions  $\varphi, \varphi_1$.

The next section contains a number of technical lemmas concerning Musielak--Orlicz spaces and multipliers. 
Consequently, we are ready to prove the representation theorem in the third section. Finally, the last section is devoted to discussion on factorization and differences between Orlicz and Musielak--Orlicz cases. In particular, we give an example showing that inequality  $\varphi_1^{-1}(\varphi\ominus\varphi_1)^{-1}\succ\varphi^{-1}$ is not necessary condition for factorization of Musielak--Orlicz spaces, unlike in the Orlicz spaces case (cf. \cite[Theorem 2]{LT17}). 

 %Young conjugate function of $\varphi$ with respect to $\varphi_1$ and its truncated version, both suitable for Orlicz-Musielak functions setting. %In case of both $\varphi$, $\varphi_1$ being Orlicz functions our definition of  Young conjugate coincides with classical one.

%In the next Section we prove series of Lemmas, witch shows that on small sets the Musielak--Orlicz Spaces behaves simillary to the Orlicz Spaces. This behavior alows us to employ some methods from \cite{LT17} and \cite{DR00} to prove main theorem that $M(L^{\varphi_1},L^{\varphi})=L^{\varphi\ominus\varphi_1}$.

%In last section we behaviour of the pointwise multiplies and factorization. Example \ref{przyklad} shows that, unlike Orlicz Spaces case, inequality  $\varphi_1^{-1}(\varphi\ominus\varphi_1)^{-1}\varphi^{-1}$ is not necessary condition for factorization.

\section{Notation and preliminaries}
Trough the paper we will assume that  $(\Omega,\Sigma,\mu)$ is a $\sigma$-finite, complete measure space. For a given set $A\in\Sigma$ we will denote  the non-atomic part and purely atomic part of $A$ by $A^c$ and $A^a$, respectively. Notice, that there may be at most countably many atoms in $\Omega$, since the measure space is $\sigma$-finite. To simplify the notion, we will assume that all atoms in $\Omega$, if there are any, are singletons. 
%When $\omega$ is an atom we will write $\omega\in\Omega^a$ and use the convention that $\omega$ will always denote atoms.

Let $L^0=L^0(\Omega,\Sigma,\mu)$ be the space of classes of equivalence of $\mu$-measurable, real valuable  and  $\mu$-a.e. finite functions.  A Banach space $X\subset L^0$ is called the {\it  Banach ideal space} %(or {\it  Banach lattice}) 
 if it satisfies the ideal property, i.e. $x\in L^0,y\in X$ and $|x|\leq |y|$ implies $x\in X$ and $\norm{x}_X\leq \norm{y}_X$ ($|x|\leq |y|$ means that $|x(t)|\leq |y(t)|$ for $\mu$-a.e. $t\in \Omega$).

For $x\in L^0$ we define its support as $\supp(x):=\{t\in\Omega:x(t)\neq 0\}.$
A support $\supp X$ of a Banach ideal space $X$ is defined as a measurable subset of $\Omega$ such that:\\ 
i) for each $x\in X$ there is $A\in \Sigma$ with $\mu(A)=0$ such that $\supp(x)\subset \supp X\cup A$,\\
ii) there is $x\in X$ such that $\mu(\supp X- \supp(x))=0$.\\
Notice that according to the above definition $\supp X$ is not unique, thus we rather write a support, than the support of $X$.

For any measurable $F\subset\Omega$ and a Banach ideal space $X$ we define 
$$
X[F]:=\{x\in X: \mu(\supp(x)\setminus F)=0\} {\rm \ with\ the\ norm\ }\norm{x}_{X[F]}=\|x\chi_F\|_{X}.
$$
%with the norm $\norm{x}_{X[F]}=\norm{x\chi_F}_{X}$. 
Given  a Banach ideal space $X$ on $\Omega$ and a positive measurable weight function $v$, the weighted space $X(v)$ is defined as
\[
X(v):=\{x\in L^0:xv\in X\} {\rm \ with\ the\ norm\ }\|x\|_{X(v)}=\|xv\|_X. 
\]

Writing $X=Y$ for two Banach ideal spaces $X,Y$ we mean that they are equal as set, but norms are just equivalent. Recall also that for Banach ideal spaces $X,Y$ the inclusion $X\subset Y$ is always continuous, i.e. there is $c>0$ such that $\norm{x}_Y\leq c\norm{x}_X$ for each $x\in X$. 

A Banach ideal space $X$ satisfies the {\it Fatou property} ($X\in (FP)$ for short) if for each sequence $(x_n)\subset X$ satisfying $x_n\uparrow x$ $\mu$-a.e. and $\sup_n\|x_n\|_X<\infty$, there holds
 $x\in X$ and $\norm{x}_X\leq \sup_n\norm{x_n}_X$. 

Given  two Banach ideal spaces $X,Y$ over the same measure space $(\Omega,\Sigma,\mu)$ we define their {\it pointwise product space} 
\[
X\odot Y=\{x\cdot y\in L^0: x\in X, y\in Y\},
\]
with a quasi--norm 
\[
\norm{z}_{X\odot Y}=\inf\{\norm{x}_X\norm{y}_Y:z=xy\}. 
\]

If additionally $\supp X =\Omega$, then {\it the space of pointwise multipliers} from $X$ to $Y$ is defined as
\[
M(X,Y)=\{y\in L^0: xy\in Y\ {\rm for\ all\ }y\in X\},
\]
with the natural operator norm 
\[
\norm{y}_{M}=\sup_{\norm{x}_X\leq1}\norm{xy}_Y. 
\]
When there is no risk of confusion we will just write $\|\cdot\|_M$ for the norm of $M(X,Y)$. 
%If  $X,Y$ two given Banach ideal spaces over the same measure space $(\Omega,\Sigma,\mu)$ 
%Such  space may be trivial, for example $M(L^p,L^q)=\{0\}$ when $p>q$, and therefore it need not be a Banach ideal space in the sense of above definition. Anyhow, it is a Banach space with the ideal property (see for example \cite{MP89}).
If Banach ideal spaces $X$ and $Y$ have the Fatou property then both spaces $M(X,Y)$ and $ X\odot Y$ have the Fatou property \cite{MP89, KLM13, KLM14}.

We will need the following easy observation concerning the space of pointwise multipliers. Let $A,B\subset \Omega$ be measurable sets such that $A\cup B=\Omega$. Given a Banach ideal space $X$ over $\Omega$, we can decompose it as
\[
X=X[A]\oplus X[B], 
\]
with the (equivalent) norm given by $\|x\|_{X[A]\oplus X[B]}=\|x\chi_A\|_{X[A]}+\|x\chi_B\|_{X[B]}$. It is easy to see that the space of pointwise multipliers respects such a ``decomposition'', i.e. $M(X,Y)$ may be written as follows 
\begin{equation}\label{dec}
M(X,Y)=M(X[A]\oplus X[B],Y[A]\oplus Y[B])=M(X[A],Y[A])\oplus M(X[B],Y[B]).
\end{equation}
In another words, determining the space of pointwise multipliers between two Banach ideal spaces, we may determine it on $A$ and on $B$ separately. 

A function $\varphi:[0,\infty)\to [0,\infty]$ will be called the {\it Young function} if it satisfies $\varphi(0)=0$,  $\lim\limits_{u\rightarrow\infty}\varphi(u)=\infty$ and is convex on $[0,b_{\varphi})$ (or on $[0,b_{\varphi}]$ when $\varphi(b_{\varphi})<\infty$), where 
$$
b_{\varphi}=\sup\{u\geq 0:\varphi(u)<\infty\}.
$$ 
We point out here that we allow $\varphi(u)=\infty$ for each $u>0$. In such a case the corresponding Orlicz space $L^{\varphi}$  contains only the zero function. 

Given a measure space $(\Omega,\Sigma,\mu)$, a function $\varphi:\Omega\times[0,\infty)\to [0,\infty]$ is called the {\it Musielak--Orlicz function} if the following conditions hold:
\begin{enumerate}
\item $\varphi(t,\cdot)$ is a Young function for $\mu$-a.e. $t\in\Omega$,
\item $\varphi(\cdot,u)$ is a measurable function for each $u\in[0,\infty)$.
\end{enumerate}
%We will need the following parameters\[a_{\varphi}=\sup\{t\geq 0:\varphi(t)=0\} {\rm\ and\ } b_{\varphi}=\sup\{t\geq 0:\varphi(t)<\infty\}.\]
Let $\varphi$ be a Musielak--Orlicz function. We define the convex modular $I_\varphi$ on $L^0$ as
\[
I_{\varphi}(x)=\int_{\Omega}\varphi(t,|x(t)|)d\mu(t).
\]
The {\it Musielak--Orlicz space} $L^{\varphi}$ is defined as
\[
L^{\varphi}=\{x\in L^0:I_{\varphi}(\lambda x)<\infty {\rm\ for\ some}\ \lambda>0\}
\]
and is equipped with the Luxemburg-Nakano norm 
\[
\norm{x}_{\varphi}=\inf\{\lambda>0:I_{\varphi}( \frac{x}{\lambda})\leq 1\}.
\]
It is  known that Musielak--Orlicz spaces have the Fatou property. % \textcolor{green}{cyt?}
Moreover, it follows immediately from the definition, that $\supp L^{\varphi}=\{t\in \Omega: b_{\varphi}(t)>0\}$ (up to a set of measure zero). In case, $\varphi$ does not depend on $t$, the Musielak--Orlicz space $L^{\varphi}$ is just the Orlicz space.

For a given Musielak--Orlicz function $\varphi$ we define two useful (functions) parameters
$$a_{\varphi}(t)=\sup\{u\geq 0:\varphi(t,u)=0\},$$
$$b_{\varphi}(t)=\inf\{u\geq 0:\varphi(t,u)=\infty\}.$$
It is known, that both $a_{\varphi}$ and $b_{\varphi}$ are measurable \cite[Proposition 5.1]{CH96}.

Furthermore, we define the right-continuous inverse at point $t\in \Omega$ 
$$
\varphi^{-1}(t,u):=\inf\{v\geq 0:\varphi(t,v)> u\}. 
$$
Properties of $\varphi^{-1}$ for a Young function $\varphi$ have been collected in \cite[Lemma 3.1]{KLM13}.

The following basic relation between the norm and the modular will be used frequently through the paper
\begin{equation}\label{nm}
\|x\|_{\varphi}\leq 1 \Rightarrow I_{\varphi}(x)\leq\norm{x}_{\varphi},
\end{equation}
 for $x\in L^{\varphi}$ %(see for example \cite[Theorem 1.1]{Ma89}). More information on Musielak--Orlicz and Orlicz spaces spaces can be found in \cite{Mu83,KT90,Ka14,KK15}.
(see  \cite[Theorem 1.1]{Ma89}). More information on Musielak--Orlicz and Orlicz spaces can be found for example in \cite{Mu83,KT90,Ka14,KK15}.

\section{Auxiliary results}

Recall that our goal is to describe the space of pointwise multipliers $M(L^{\varphi_1},L^{\varphi})$ between two Musielak--Orlicz spaces and thus we will operate on two  Musielak--Orlicz functions  $\varphi, \varphi_1$, both defined over the same measure space $\Omega$. The result will be given in terms of the third  Musielak--Orlicz function $\varphi\ominus\varphi_1$ -  the Young conjugate of $\varphi_1$ with respect to $\varphi$.
In order to define it we need to introduce the following decomposition of the continuous part of the domain $\Omega$ depending on  behaviour of both $\varphi, \varphi_1$. Let $\varphi, \varphi_1$ be two Musielak--Orlicz functions. We define the following sets: 
\begin{align*}
\Omega_{0,0}&:=\{t\in\Omega^c:b_{\varphi_1}(t)=b_{\varphi}(t)=\infty\},\\
\Omega_{0,\infty}&:=\{t\in\Omega^c:b_{\varphi_1}(t)=\infty,\ b_{\varphi}(t)<\infty\},\\
\Omega_{\infty,0}&:=\{t\in\Omega^c:0<b_{\varphi_1}(t)<\infty,\ b_{\varphi}(t)=\infty\},\\
\Omega_{\infty,\infty}&:=\{t\in\Omega^c:0<b_{\varphi_1}(t)<\infty,\ b_{\varphi}(t)<\infty\},\\
\Omega_{\infty}&:=\Omega_{\infty,\infty}\cup\Omega_{\infty,0}.
\end{align*}
%Moreover, if $A$ is one of the above sets, then $A^c$ denotes its continuous part, i.e. $A^c:=\Omega^c\cap A$. 

Given  two Musielak--Orlicz functions $\varphi, \varphi_1$ over the same measure space $\Omega$, 
 the {\it Young conjugate of $\varphi_1$ with respect to $\varphi$} is defined as
$$\varphi\ominus\varphi_1(t,u) :=
 \begin{cases}
  \sup\{ \varphi(t,su)-\varphi_1(t,s): 0\leq s < b_{\varphi_1}(t)\} &\mbox{if } t\in\Omega^c, \\
 
  \sup\{\varphi(t,su)-\varphi_1(t,s): 0\leq s \leq \min\{1/\varphi^{-1}(t,\frac{1}{\mu(\{t\})}),\frac{b_{\varphi_1}(t)}{2}\}\} &\mbox{if } t\in\Omega^a. \\
  \end{cases} $$ 
 %\sup\{\varphi(t,su)-\varphi_1(t,s): 0\leq s \leq b_{\varphi_1}(t)\} &\mbox{if } t\in\Omega^c_{\infty} \
Observe  firstly that such defined function $\varphi\ominus\varphi_1$ satisfies  $b_{\varphi\ominus\varphi_1}(t)>0$ for each  $t\in \Omega^a$. 
Moreover, it is easy to see, that for $t\in \Omega_{0,\infty}$ 
  $$\varphi\ominus\varphi_1(t,u)=
  \begin{cases}
  0 &\mbox{if } u=0,\\
  \infty &\mbox{if } u>0.
  \end{cases}$$
In consequence,
 \begin{equation}
\supp(L^{\varphi\ominus\varphi_1})=(\Omega_{0,0}\cup\Omega_{\infty}\cup\Omega^a)\cap\supp(L^{\varphi}) \label{suportlminus}
\end{equation} 

%%%%%%%%%%%%%%%%%%%%%%%%%%%%%%%%%%%%%%%%%%%%%%%%%%%%%%%%
%{\color{red} NOWE}
It may be instructive to realize what is $\varphi\ominus\varphi_1$, when $\varphi,\varphi_1$ are Nakano functions.
\begin{example}\label{exam1}
Let $p,q:\Omega\rightarrow [1,\infty)$ be two measurable functions and define $\varphi(t,u)=\frac{1}{q(t)}u^{q(t)}$,  $\varphi_1(t,u)=\frac{1}{p(t)}u^{p(t)}$  for $t\in\Omega, u\geq 0$. Assume that $q(t)\leq p(t)$ for a.e. $t\in\Omega$. One can easily calculate that
$$
\varphi\ominus\varphi_1(t,u)=\frac{1}{r(t)}u^{r(t)}
$$
where 
$\frac{1}{p(t)}+\frac{1}{r(t)}=\frac{1}{q(t)}$ for a.e. $t\in\Omega$. 
%It is known (\cite[Theorem 4.8]{Kar18}) that
%$$M(L^{\varphi_1},L^{\varphi})=L^{\varphi\ominus\varphi_1}.$$
\end{example}
%%%%%%%%%%%%%%%%%%%%%%%%%%%%%%%%%%%%%%%%%%%%%%%%%%%%%

Finally, we shall see that $\varphi\ominus\varphi_1$ satisfies assumptions of Musielak--Orlicz functions. 

\begin{lemma}\label{MOminus}
Given  two Musielak--Orlicz functions $\varphi, \varphi_1$ over the same measure space $\Omega$, the function $\varphi\ominus\varphi_1$ is also the Musielak--Orlicz function on $\Omega$. 
\end{lemma}
\proof
It is already known that $\varphi\ominus\varphi_1(t,\cdot)$ is a Young function for a.e. $t\in \Omega$ (see \cite{KLM13, LT17}). Thus we need only to prove measurability of $\varphi\ominus\varphi_1(\cdot,u)$ for each $u\geq 0$. Each component of $\Omega$ may be considered separately. We will explain only situation on $\Omega_{\infty}$, since the remaining cases are either simpler, or evident. Firstly we observe that for each $u\geq 0$ and $t\in \Omega_{\infty}$
\[
 \varphi\ominus\varphi_1(t,u)=\sup\{ \varphi(t,su)-\varphi_1(t,s): 0\leq s < b_{\varphi_1}(t)\}
\]
\[
=\sup\{ \varphi(t,ub_{\varphi_1}(t)v)-\varphi_1(t,b_{\varphi_1}(t)v): 0\leq v < 1\}
\]
\[
=\sup\{ \varphi(t,ub_{\varphi_1}(t)v)-\varphi_1(t,b_{\varphi_1}(t)v): v\in \mathbb{Q}\cap [0,1)\}.
\]
However, functions $\varphi(\cdot,ub_{\varphi_1}(\cdot)v)$ and $\varphi_1(\cdot,b_{\varphi_1}(\cdot)v)$ are measurable by properties of Musielak--Orlicz functions. In consequence, $\varphi\ominus\varphi_1(\cdot,u)$ is measurable, as the supremum of countable collection of measurable functions.
\endproof

In the proof of the main theorem, we are going to imitate  inductive argument used in \cite{DR00} and in \cite{LT17}. In order to do it we need a kind of decomposition of the measure space $\Omega$. The following two lemmas provide it. 

\begin{lemma}\label{bfiinfty}
Let $(\Omega,\Sigma,\mu)$ be non-atomic and let $\varphi$ be a Musielak--Orlicz function such that $b_{\varphi}(t)=\infty$ for $\mu$-a.e. $t\in\Omega$. For each $a>0$ there exists a sequence of pairwise disjoint measurable sets $(A_n)$ such that $\bigcup\limits_{n\in\mathbb{N}}A_n=\Omega$ and $$\norm{\chi_{A_n}}_{\varphi}\leq\frac{1}{a}$$ for every $n\in\mathbb{N}$.
\end{lemma}
\begin{proof}
Fix $a>0$. Define the sets 
$$B_n:=\{t\in\Omega:n-1\leq\varphi(t,a)< n\}$$
for $n\in \mathbb{N}$.
Evidently, each $B_n$ is measurable, since the function $\varphi(\cdot,a)$ is measurable. Moreover  $\bigcup\limits_{n\in\mathbb{N}}B_n=\Omega$ and $(B_n)$ is a sequence of pairwise disjoint sets. Since we operate on a non-atomic measure space, each $B_n$ may be divided further into a sequence (finite or not) of pairwise disjoint sets $(C^n_j)_{j\in I_n}$ such that $\bigcup\limits_{j\in I_n}C^n_j=B_n$ and $\mu(C^n_j)\leq\frac{1}{n}$ for each $j\in I_n$. In consequence, we have for $n\in\mathbb{N}$ and $j\in I_n$
\begin{align*}
\I{a\chi_{C^n_j}}{\varphi}&=\int_{C^n_j}\varphi(t,a)d\mu(t)\\
&\leq\mu(C^n_j)\sup\limits_{t\in C^n_j} \varphi(t,a)
\leq 1.
\end{align*}
It follows that 
$$\norm{\chi_{C^n_j}}_{\varphi}\leq\frac{1}{a},$$
for every $n\in\mathbb{N}$ and $j\in I_n$. Finally, we get the desired sequence $(A_n)$ just by rearranging the (doubly indexed) sequence $(C^n_j)$.
\end{proof}

\begin{lemma}\label{bfiskon}
Let $(\Omega,\Sigma,\mu)$ be non-atomic and let $\varphi$ be a Musielak--Orlicz function such that $0<b_{\varphi}(t)<\infty$ for $\mu$-a.e. $t\in\Omega$. There exists a sequence of pairwise disjoint measurable sets $(A_n)$ such that $\bigcup\limits_{n\in\mathbb{N}}A_n=\Omega$ and for each $n\in\mathbb{N}$
$$
\norm{\chi_{A_n}}_{\varphi}\leq\frac{2}{\esssup\limits_{t\in A_n}\{b_{\varphi}(t)\}}.
$$ 
\end{lemma}

\begin{proof}
For each $k\in\mathbb{Z}$ define
$$B_k:=\{t\in\Omega:2^{k-1}<b_{\varphi}(t)\leq 2^{k}\}.$$
Evidently, sets $B_k$ are measurable, since $b_{\varphi}$ is a measurable function.
Next, for each $k\in\mathbb{Z}$ and $n\in\mathbb{N}$ we define
$$
B_{k,n}:=\{t\in B_k:n-1\leq\varphi(t,2^{k-1})< n\}.
$$
Then the doubly indexed sequence $(B_{k,n})$ consists of pairwise disjoint measurable sets such that $\bigcup\limits_{n\in\mathbb{N},k\in\mathbb{Z}}B_{k,n}=\Omega$. Denote $$I:=\{(k,n)\in\mathbb{Z}^2:B_{k,n}\neq\emptyset\}.$$ 
For each $(k,n)\in I$ we can further decompose $B_{k,n}$  into a (finite or not) sequence $(C^{k,n}_j)_{j\in I_{k,n}}$  of pairwise disjoint measurable sets in such a way that $\bigcup\limits_{j\in I_{k,n}}C^{k,n}_j=B_{k,n}$ and $\mu(C^{k,n}_j)\leq\frac{1}{n}$ for each $j\in I_{k,n}$. Finally, for every $(k,n)\in I$ and $j\in I_{k,n}$ we have
\begin{align*}
\I{2^{k-1}\chi_{C^{k,n}_j}}{\varphi}&=\int_{C^{k,n}_j}\varphi(t,2^{k-1})d\mu(t)\\
&\leq\mu(C^{k,n}_j)\sup\limits_{t\in C^{k,n}_j} \varphi(t,2^{k-1})
\leq 1.
\end{align*}
In consequence, 
$$\norm{\chi_{C^{k,n}_j}}_{\varphi}\leq\frac{1}{2^{k-1}}\leq\frac{2}{\esssup\limits_{t\in B_{k,n}}\{b_{\varphi}(t)\}}\leq\frac{2}{\esssup\limits_{t\in C^{k,n}_j}\{b_{\varphi}(t)\}}.$$
Similarly as before, the desired sequence is obtained after rearranging  the (triple indexed) sequence $(C^{k,n}_j)$.
\end{proof}

\begin{fact}\label{infty} 
If a Musielak--Orlicz function $\varphi$ is such that $b_{\varphi}(t)<\infty$ for a.e. $t\in\Omega$, then
\[
L^{\varphi}\subset L^{\infty}(1/b_{\varphi}).
\]
\end{fact}
\begin{proof}
Let $0\leq y\notin L^{\infty}(1/b_{\varphi})$. For each $n\in \mathbb{N}$ we define sets
$$
A_n=\{t\in\Omega : n\leq \frac{y(t)}{b_{\varphi}(t)}\}.
$$ 
Then there is  $N\in\mathbb{N}$ such that  $\mu(A_n)>0$ for $n\geq N$. Fix $a>0$ and choose $n\geq N$ satisfying $an>2$. We can see that 
$$
nb_{\varphi}\chi_{A_n}\leq y.
$$ 
In consequence,
$$
\I{ay}{\varphi}\geq\I{anb_{\varphi}\chi_{A_n}}{\varphi}\geq\I{2b_{\varphi}\chi_{A_n}}{\varphi}=\infty.
$$
Since $a>0$ was arbitrary we conclude that $y\notin L^{\varphi}.$
\end{proof}

\begin{lemma}\label{linfty}
Let  $(\Omega,\Sigma,\mu)$ be non-atomic and let $\varphi, \varphi_1$ be two Musielak--Orlicz functions such that $0<b_{\varphi_1}(t)<\infty$ and $0< b_{\varphi}(t)<\infty$ for $\mu$-a.e. $t\in\Omega$. Then 
$$
M(L^{\varphi_1},L^{\varphi})\subset L^{\infty}(b_{\varphi_1}/b_{\varphi}).
$$ 
%where $v(t):=\frac{b_{\varphi_1}(t)}{b_{\varphi}(t)}$, $t\in \Omega$.
\end{lemma}

\begin{proof}
Let $0\leq y\notin L^{\infty}(v)$, where $v(t):=\frac{b_{\varphi_1}(t)}{b_{\varphi}(t)}$. For each $n\in \mathbb{N}$ we define 
$$
A_n=\{t\in\Omega :n\leq y(t)v(t)< n+1\}.
$$ 
Then there exist infinitely many $n\in\mathbb{N}$ for which $\mu(A_n)>0$. Denote the set of such $n$'s by $I$.
Next, since $\Omega$ is non-atomic, for each $n\in I$ there is $B_n\subset A_n$ such that $\mu(B_n)>0$ and 
$$
\int_{B_n}\varphi_1(t,\frac{b_{\varphi_1}(t)}{2})d\mu(t)\leq\frac{1}{2^n}.
$$
We define 
$$f(t):=\sum\limits_{n=1}^{\infty}\frac{b_{\varphi_1}(t)}{2}\chi_{B_n}.$$
Then
$$\I{f}{\varphi_1}=\int\varphi_1(t,f(t))d\mu(t)=\sum\limits_{n=1}^{\infty}\int_{B_n}\varphi_1(t,\frac{b_{\varphi_1}(t)}{2})d\mu(t)\leq 1.$$
It means that $f\in L^{\varphi_1}$ and $\norm{f}_{\varphi_1}\leq 1$. However,
$$
y(t)f(t)\geq\frac{1}{2}y(t)b_{\varphi_1}(t)\geq\frac{n}{2}b_{\varphi}(t)\quad \text{ for a.e. } t\in B_n,
$$
which implies that $yf\notin L^{\varphi}$, since $L^{\varphi}\subset L^{\infty}(\frac{1}{b_{\varphi}})$ by Fact \ref{infty}. Consequently, $y\notin M(L^{\varphi_1},L^{\varphi})$ and the proof is finished.
\end{proof}

\begin{lemma}\label{suppm}
Suppose  $(\Omega,\Sigma,\mu)$ is non-atomic and let $\varphi, \varphi_1$ be Musielak--Orlicz functions such that $\supp L^{\varphi_1}=\Omega$. Then 
$$
 \mu(\supp M(L^{\varphi_1},L^{\varphi})\setminus\Omega_{0,0}\cup\Omega_{\infty})=0.
 $$
\end{lemma}

\begin{proof}
%We will show that $$M(L^{\varphi_1},L^{\varphi})(\Omega_{0,\infty})=\{0\}.$$
We need only to show that $\mu(\Omega_{0,\infty}\cap \supp M(L^{\varphi_1},L^{\varphi}))=0$. Suppose, for a contrary, there exists $A\subset\Omega_{0,\infty}$ such that $\mu(A)>0$ and $\chi_A\in M(L^{\varphi_1},L^{\varphi})$.  Let $C\subset A$ be chosen in such a way that  $\mu(C)>0$ and $ \inf\limits_{t\in C}b_{\varphi}(t)=\delta>0$. 
From Lemma \ref{bfiinfty}  it follows that for each $n\in\mathbb{N}$ there exists $A_n\subset C$ such that $\mu(A_n)>0$ and 
$$
\norm{\chi_{A_n}}_{\varphi_1}\leq\frac{1}{n}.
$$
Moreover, by Fact \ref{infty}, we know that $L^{\varphi}[\Omega_{0,\infty}]\subset L^{\infty}(\frac{1}{b_{\varphi}})[\Omega_{0,\infty}]$ with some inclusion constant $c>0$. It means
\begin{align*}
\norm{\chi_{A_n}}_{\varphi}&\geq c^{-1}\norm{\chi_{A_n}}_{L^{\infty}(\frac{1}{b_{\varphi}})}\\
&\geq c^{-1}\sup\limits_{t\in A_n}\frac{1}{b_{\varphi}(t)}\\
&=\frac{1}{c\inf\limits_{t\in A_n}b_{\varphi}(t)}\\
&\geq\frac{1}{c\inf\limits_{t\in C}b_{\varphi}(t)}=\frac{1}{c\delta}.
\end{align*}
Finally, for each $n\in\mathbb{N}$ define $x_n:=n\chi_{A_n}$. Then $\norm{x_n}_{L^{\varphi_1}}\leq 1$ and it follows
$$
\norm{\chi_A}_{M}\geq\norm{x_n\chi_A}_{\varphi}=\norm{n\chi_{A_n}}_{\varphi}\geq\frac{n}{c\delta},
$$
for each $n\in\mathbb{N}$.
In consequence, $\chi_A\not \in M(L^{\varphi_1},L^{\varphi})$ which contradicts our assumption. 
\end{proof}

Of course, the supremum in definition of function $\varphi\ominus \varphi_1$ need not be attained. To avoid such a situation,  we introduce a truncated version of $\varphi\ominus \varphi_1$ (cf. \cite[Definition 1]{LT17}). 
Namely, for $a>0$ we define the function  $\varphi\ominus_a \varphi_1$ in the following way
 $$\varphi\ominus_a\varphi_1(t,u) :=
 \begin{cases}
  \sup\{ \varphi(t,su)-\varphi_1(t,s): 0\leq s\leq a \} &\mbox{if } t\in\Omega_{0,0}\cup\Omega_{0,\infty}  \\
  \sup\{\varphi(t,su)-\varphi_1(t,s): 0\leq s \leq \frac{a}{a+1}b_{\varphi_1}(t)\} &\mbox{if } t\in\Omega_{\infty} \\
  \sup\{\varphi(t,su)-\varphi_1(t,s): 0\leq s \leq \min\{1/\varphi^{-1}(t,\frac{1}{\mu(\{t\})}),\frac{b_{\varphi_1}(t)}{2}\}\} &\mbox{if } t\in\Omega^a \\
  \end{cases} $$
Using exactly the same reasoning as in the proof of Lemma \ref{MOminus} one can see that $\varphi\ominus_a\varphi_1$ is the Musielak--Orlicz function over $\Omega$. Furthermore, it is easy to see that 
\begin{equation}\label{B}
b_{\varphi\ominus_a\varphi_1}(t)=\frac{(a+1)b_{\varphi}(t)}{ab_{\varphi_1}(t)}
\end{equation}
for $t\in \Omega_{\infty}$.

\begin{lemma}\label{mierzalnosc}
Let  $(\Omega,\Sigma,\mu)$ be non-atomic and let $\varphi, \varphi_1$ be Musielak--Orlicz functions such that $\supp L^{\varphi_1}=\Omega$. If $A\subset\supp L^{\varphi}\setminus\Omega_{\infty,0}$  is a set of positive measure and numbers $a>1,u>0$ satisfy
$\varphi\ominus_a\varphi_1(t,\frac{3}{2}u)<\infty$  for a.e. $t\in A$, then the function $x:A\to \mathbb{R}_+$, defined by
$$
x(t):=\max\{0\leq v\leq \min\{a, \frac{a}{a+1}b_{\varphi_1}(t)\}:\varphi_1(t,v)+\varphi\ominus_a\varphi_1(t,u)=\varphi(t,uv) \},
$$ 
is measurable. 
\end{lemma}
\begin{proof}

Without loss of generality we may assume that $\varphi_1(t,\cdot),\varphi(t,\cdot)$ are Young functions for each $t\in A$. Fix $u>0$ and $a>1$ satisfying
\begin{equation}\label{A}
\varphi\ominus_a\varphi_1(t,\frac{3}{2}u)<\infty  {\rm \ for\ a.e.\ } t\in A
\end{equation}
and let $x$ be like in the statement.
Let $(r_k)$ be a dense sequence in $[0,a]$.  For each $k,n\in \mathbb{N}$ define
$$
B^n_k:=\{t\in A: r_k\leq \frac{a}{a+1}b_{\varphi_1}(t),\ \varphi_1(t,r_k)+\varphi\ominus_a\varphi_1(t,u)-\varphi(t,ur_k)<1/n \}
$$
and 
$$
q^n_k:=r_k\chi_{B^n_k}.
$$
Just notice that by the definition of $\varphi\ominus_a\varphi_1$
$$0\leq\varphi_1(t,v)+\varphi\ominus_a\varphi_1(t,w)-\varphi(t,wv)$$
for a.e. $t\in\Omega$ and $w,v\geq 0$. Therefore,
$$\varphi(t,ur_k)<\infty,$$
because  for every $k\in\mathbb{N}$ we have $\varphi_1(t,r_k)<\infty$ and $\varphi\ominus_a\varphi_1(t,u)<\infty$.
Of course, functions $q^n_k$ are measurable, since sets $B^n_k$ are measurable. We will show that 
\begin{equation}
x=\limsup_{k,n\to \infty} q^n_k.
\end{equation}
Firstly we will explain the inequality $\limsup_{k,n\to \infty} q^n_k\leq x$. Suppose, for a contradiction, that for some $t_0\in A$ and some $\delta>0$ there holds
$$
\limsup_{k,n\to \infty} q^n_k(t_0)>x(t_0)+\delta.
$$ 
%where $x(t_0)+\delta<a$. 
This implies that there is a (singly-indexed) sequence $(q^{n_i}_{k_i})$ such that $\min\{a,\frac{a}{a+1}b_{\varphi_1}(t_0)\}\geq q^{n_i}_{k_i}(t_0)>x(t_0)+\delta$, $n_i,k_i\to \infty$ and 
\begin{equation}\label{l1}
\varphi_1(t_0,u_{q^{n_i}_{k_i}}(t_0))+\varphi\ominus_a\varphi_1(t_0,u)-\varphi(t_0,u_{q^{n_i}_{k_i}}(t_0))<1/n_i
\end{equation}
for each $i=1,2,3,...$. On the other hand, there is a subsequence  $(q_{j}):=(q^{n_{i_j}}_{k_{i_j}})$ of $(q^{n_i}_{k_i})$ and $q_0>x(t_0)$ such that $q_j(t_0)\to q_0$. However, by (\ref{l1}) and continuity of respective functions, we get 
$$
\varphi_1(t_0,q_{0})+\varphi\ominus_a\varphi_1(t_0,u)-\varphi(t_0,uq_{0})=0,
$$
which contradicts maximality of $x(t_0)$ and proves inequality $\limsup_{k,n\to \infty} q^n_k\leq x$.

To see the opposite inequality fix $t\in A$ and  denote  
$$
C_n:=\{0\leq v\leq \min\{a,\frac{a}{a+1}b_{\varphi_1}(t)\}:\varphi_1(t,v)+\varphi\ominus_a\varphi_1(t,u)-\varphi(t,uv)<1/n\}.
$$
We see that sets $C_n$ are open and non-empty, since $x(t)\in C_n$ for each $n$. Therefore, one can select a sequence $(r_{n_i})$ such that $r_{n_i}\in C_{i}$ and $r_{n_i}\to x(t)$. Then $t\in B^i_{n_i}$ for each $i=1,2,3,...$ and, consequently,
$$
x(t)\leq \limsup_{k,n\to \infty} q^n_k(t),  
$$
which finally proves measurability of $x$.
\end{proof}

\section{Pointwise multipliers}
 
\begin{theorem}\label{th1}
Let $\varphi, \varphi_1$ be Musielak--Orlicz functions over a measure space $(\Omega,\Sigma,\mu)$  and  assume that $ \supp L^{\varphi_1}=\Omega$. Then
$$M(L^{\varphi_1},L^{\varphi})=L^{\varphi\ominus\varphi_1}.$$
\end{theorem}

\begin{proof}
Without loss of generality we can assume that $\supp(L^{\varphi})=\Omega$, since 
$$
M(L^{\varphi_1},L^{\varphi})[\Omega\setminus\supp(L^{\varphi})]=\{0\}=L^{\varphi\ominus\varphi_1}[\Omega\setminus\supp(L^{\varphi})],
$$
where the second equality follows from (\ref{suportlminus}).
The proof of inclusion 
$$
L^{\varphi\ominus\varphi_1}\subset M(L^{\varphi_1},L^{\varphi})
$$
is the same as in the case of Orlicz spaces and we omit it (see for example \cite[Lemma 6]{LT17}).

We only need to prove the remaining inclusion
$$
M(L^{\varphi_1},L^{\varphi})\subset L^{\varphi\ominus\varphi_1}
$$ 

%=\frac{a}{a+1}b^{-1}_{\varphi\ominus_a\varphi_1}(t)$ for $a>1$.
%From Lemma \ref{linfty} we have $M(L^{\varphi_1},L^{\varphi})(\Omega_{\infty,\infty})\xhookrightarrow{c} L^{\infty}(v)(\Omega_{\infty,\infty})$, where $v(t)=\frac{b_{\varphi_1}(t)}{b_{\varphi}(t)}.$

Let $0\leq y\in M(L^{\varphi_1},L^{\varphi})$ be a simple function %of the form $y=\sum_k a_k\chi_{B_k}$ and 
such that $\norm{y}_{M}\leq\frac{1}{4c}$, where $c\geq 1$ is the constant of inclusion 
$$
M(L^{\varphi_1},L^{\varphi})[\Omega_{\infty,\infty}]\subset  L^{\infty}(b_{\varphi_1}/b_{\varphi})[\Omega_{\infty,\infty}]
$$
(cf. Lemma \ref{linfty}). We will show that
\begin{equation}
I_{\varphi\ominus_a\varphi_1}(y)\leq 1\label{claim}
\end{equation}
for every $a>1$. % From this inequality it will follow that $y\in L^{\varphi\ominus\varphi_1}$ and $$\norm{y}_{\varphi\ominus\varphi_1}\leq 1.$$ \\
%Fix . 
 To prove this inequality, for each $a>1$ we will construct a function $x(t)$ on $\Omega$ and  a family of pairwise disjoint sets $(A_n)$ satisfying:
\begin{enumerate}
\item $\varphi\ominus_a\varphi_1(t,y(t))=\varphi(t,x(t)y(t))-\varphi_1(t,x(t))$ for a.e $t\in\Omega$,
\item $\norm{xy\chi_{A_n}}_{\varphi}\leq\frac{1}{2}$ for each $n\in \mathbb{N}$,
\item $\supp(M(L^{\varphi_1},L^{\varphi}))\subset\bigcup\limits_{n\in\mathbb{N}}A_n$,
\item $x\in L^{\varphi_1}$ and $\norm{x}_{\varphi_1}\leq 1.$
\end{enumerate} 

%Firstly we will construct the function.\\ %$\Omega_{\infty}, \Omega_{0,0}$ and $\Omega^a$.\\
Let $a>1$.
Since $y$ is a simple function we can write it in the form  
$$
y=\sum\limits_{k=0}^n b_k\chi_{B_k}+\sum\limits_{k=0}^m d_k\chi_{\{\omega_k\}},
$$
where for every $k$ we have $b_k,d_k>0$, $B_k\subset\Omega_{\infty}\cup \Omega_{0,0}$, $\mu(B_k)<\infty$ and $\omega_k$'s are atoms.
In order to construct the desired function $x$, we will apply Lemma \ref{mierzalnosc} for each $b_k$ and $B_k$. First of all we need to show that assumptions of Lemma \ref{mierzalnosc} are fulfilled, i.e. for each $0\leq k\leq n$ we have $\varphi\ominus_a\varphi_1(t,\frac{3}{2}b_k)<\infty$ for a.e. $t\in B_k$.
Let $0\leq k\leq n$. Then for a.e. $t\in B_k$ we have
$$
b_k=y(t)\leq\frac{b_{\varphi\ominus\varphi_1}(t)}{4}\leq\frac{b_{\varphi\ominus_a\varphi_1}(t)}{2},
$$
since, by Lemma \ref{linfty},
$$
\norm{yb^{-1}_{\varphi\ominus\varphi_1}\chi_{\Omega_{\infty,\infty}}}_{\infty}\leq c\norm{y\chi_{\Omega_{\infty,\infty}}}_{M}\leq\frac{1}{4}
$$
and
$$
b_{\varphi\ominus_a\varphi_1}(t)=\infty
$$
for a.e. $t\in \Omega_{0,0}\cup\Omega_{\infty,0}$. Consequently $\varphi\ominus_a\varphi_1(t,\frac{3}{2}b_k)<\infty$ for a.e. $t\in B_k$.  %(note that $b_{\varphi\ominus_a\varphi_1}=\infty$ a.e. on $\Omega_{\infty,0}$).
Thus using Lemma \ref{mierzalnosc} for the set $B_k$ and  the number $b_k$ we obtain measurable function $x_k(t)$ on $B_k$ such that
$$
\varphi\ominus_a\varphi_1(t,y(t))=\varphi(t,x_k(t)y(t))-\varphi_1(t,x_k(t))
$$ %On $\Omega^c\setminus\supp(M(L^{\varphi_1},L^{\varphi}))$ we set $x(t)=0$.
%for a.e. $t\in \Omega_{\infty}$.
and $0\leq x_k(t)\leq\min\{a, \frac{a}{a+1}b_{\varphi_1}(t)\}$  for a.e. $t\in B_k$.

Now we will consider the atomic part. For every $0\leq k\leq m$ let $c_k>0$ satisfy
$$
0\leq c_k\leq \min\{1/\varphi^{-1}(\omega_k,\frac{1}{\mu(\{\omega_k\})}),\frac{b_{\varphi_1}(\omega_k)}{2}\}\}
$$
and
$$
\varphi\ominus_a\varphi_1(\omega_k,y(\omega_k))=\varphi(\omega_k,c_ky(\omega_k))-\varphi_1(\omega_k,c_k).
$$
Such numbers exist, since the supremum in definition of $\varphi\ominus_a\varphi_1$ is taken over a compact set.

The function satisfying (i) is defined as
$$
x(t) :=
 \begin{cases}
 x_k(t) &\mbox{if } t\in B_k,\ 0\leq k\leq n,  \\
 c_k &\mbox{if } t=\omega_k,\ 0\leq k\leq m, \\
  0 &\mbox{if } t\notin \supp(y). \\
  \end{cases} 
$$
%Desired sets will be firstly constructed on $\Omega_{\infty}$, then on $\Omega_{0,0}$ and lastly on atomic part.
In the next step we will determine sets $(A_n)$ satisfying (ii) and (iii). 

We start with $\Omega_{\infty}$. By Lemma \ref{bfiskon} there exists a sequence of pairwise disjoint measurable sets $(A^1_n)$ such that $\bigcup\limits_{n\in\mathbb{N}}A^1_n=\Omega_{\infty}$ and
$$
\norm{\chi_{A^1_n}}_{\varphi_1}\leq\frac{2}{\sup\limits_{t\in A^1_n}\{b_{\varphi_1}(t)\}},
$$
for every $n\in\mathbb{N}$. Since  $0\leq x(t)< b_{\varphi_1}(t)$, we have
\begin{equation}\label{eqt1}
\norm{xy\chi_{A^1_n}}_{\varphi}\leq\sup_{t\in A^1_n}\{b_{\varphi_1}(t)\}\norm{y}_M\norm{\chi_{A^1_n}}_{\varphi_1}\leq\frac{1}{2},
\end{equation}
and therefore sets  $(A^1_n)$ satisfy (ii).\\
%Similarly as in previous case, from Lemma \ref{mierzalnosc}, there exist a measurable function $x(t)$ such that
%$$\varphi\ominus_a\varphi_1(t,y(t))=\varphi(t,x(t)y(t))-\varphi_1(t,x(t))$$
%for a.e. $t\in \Omega^c_{0,0}$ with $0\leq x(t)\leq a$.
Secondly, by Lemma \ref{bfiinfty},  there exists sequence $(A^2_n)$ of pairwise disjoint measurable sets  such that $\bigcup\limits_{n\in\mathbb{N}}A^2_n=\Omega_{0,0}$ and 
$$
\norm{\chi_{A^2_n}}_{\varphi_1}\leq\frac{1}{a}.
$$
Moreover, we have
\begin{equation}\label{eqt2}
\norm{xy\chi_{A^2_n}}_{\varphi}\leq\frac{a}{2}\norm{\chi_{A^2_n}}_{\varphi_1}\leq\frac{1}{2},
\end{equation}
because $x(t)\leq a$. 

Considering the atomic part, let's observe that for each  $\omega\in\Omega^a$
\begin{equation}\label{eqt3} 
\norm{xy\chi_{\{\omega\}}}_{\varphi}\leq\frac{1}{2\varphi^{-1}(\omega,\frac{1}{\mu(\{\omega\})})}\norm{\chi_{\{\omega\}}}_{\varphi_1}=\frac{1}{2},
\end{equation}
where the last equality follows by  $\norm{\chi_{\{\omega\}}}_{\varphi_1}=\varphi^{-1}(\omega,\frac{1}{\mu(\{\omega\})})$. Therefore, we can take atoms as desired sets.

Finally, it is enough to renumerate the sequences $(A^1_n),(A^2_n),(\{\omega\})_{\omega\in\supp(M(L^{\varphi_1},L^{\varphi}))^a}$ into one sequence $(A_n)$. 
By Lemma \ref{suppm}, 
$$
\supp(M(L^{\varphi_1},L^{\varphi}))\subset\bigcup_{n\in\mathbb{N}}A_n,
$$  
thus the construction of desired sets $(A_n)$ is finished.

It just left to show that (iv) is fulfilled, i.e.
$$
\norm{x}_{\varphi_1}\leq 1.
$$
In order to prove it, we define functions $x_n:=\sum\limits_{k=1}^n x\chi_{A_k}$ and we will inductively show that 
$$
\I{x_n}{\varphi_1}\leq \frac{1}{2}.
$$
Since $x_n\uparrow x$ a.e., from the Fatou property, it will follow that $x\in L^{\varphi_1}$ and
$$
\norm{x}_{\varphi_1}\leq \sup\limits_n\norm{x_n}_{\varphi_1}\leq 1.
$$

Firstly we need to show that for every $k\in\mathbb{N}$ there holds 
\begin{equation}\label{tez11}
\norm{x\chi_{A_k}}_{\varphi_1}\leq\frac{1}{2}.
\end{equation}
From the equality  
$$
\varphi\ominus_a\varphi_1(t,y(t))=\varphi(t,x(t)y(t))-\varphi_1(t,x(t))
$$ %and properties of $x$ 
we obtain two inequalities
\begin{eqnarray}
\varphi_1(t,x(t))\leq\varphi(t,x(t)y(t))\label{pierwsza} {\rm \ for\ a.e.\ } t\in \Omega,\\
\varphi\ominus_a\varphi_1(t,y(t))\leq\varphi(t,x(t)y(t))\label{druga}  {\rm \ for\ a.e.\ } t\in \Omega.
\end{eqnarray}
From (\ref{pierwsza}) and by inequality $\norm{xy\chi_{A_k}}_{\varphi}\leq\frac{1}{2}$ we have 
\begin{equation}\label{eqt4} 
\I{x\chi_{A_k}}{\varphi_1}=\int_{A_k}\varphi_1(t,x(t))d\mu(t)\leq\int_{A_k}\varphi(t,y(t)x(t))d\mu(t)=\I{yx\chi_{A_k}}{\varphi}\leq\frac{1}{2}
\end{equation}
for every $k\in\mathbb{N}$, where the last inequality follows from (\ref{nm}).\\
In particular, $\I{x_1}{\varphi_1}\leq \frac{1}{2}$, and we can proceed with the induction. Let $n\geq 1$ and suppose that
$$\I{x_n}{\varphi_1}\leq\frac{1}{2}.
$$
We have
$$\I{x_{n+1}}{\varphi_1}=\I{x_n}{\varphi_1}+\I{x\chi_{A_{n+1}}}{\varphi_1}\leq 1$$ and thus $\norm{x_{n+1}}_{\varphi_1}\leq 1.$
%$$\norm{x_{n+1}}_{\varphi_1}\leq\norm{x_{n}}_{\varphi_1}+\norm{x\chi_{A_{n+1}}_{\varphi_1}\leq 1$$
Similarly, as in inequality (\ref{eqt4}), we obtain
$$
\I{x_{n+1}}{\varphi_1}\leq \I{yx_{n+1}}{\varphi}\leq\frac{1}{2},
$$
by $\norm{yx_{n+1}}_{\varphi}\leq\frac{1}{2}\norm{x_{n+1}}_{\varphi_1}\leq\frac{1}{2}$. It means that (\ref{tez11}) is proved and (iv) follows.

Finally, we are ready to show that $I_{\varphi\ominus_a\varphi_1}(y)\leq 1$. We have 
$$
\norm{yx}_{\varphi}\leq\norm{y}_M\norm{x}_{\varphi_1}\leq\frac{1}{2}
$$
and from inequality (\ref{druga}) we obtain
$$
\I{y}{\varphi\ominus_a\varphi_1}=\int\varphi\ominus_a\varphi_1(t,y(t))d\mu(t)\leq\int\varphi(t,y(t)x(t))d\mu(t)= \I{yx}{\varphi}\leq 1.
$$
Clearly, $\varphi\ominus_a\varphi_1(t,y(t))\uparrow\varphi\ominus\varphi_1(t,y(t))$ for a.e. $t\in\Omega$ when $a\uparrow\infty$. Applying the Fatou lemma we have
$$
\I{y}{\varphi\ominus\varphi_1}=\int\limits_{\Omega}\varphi\ominus\varphi_1(t,y(t))d\mu(t)
\leq\liminf\limits_{a\rightarrow\infty}\int\limits_{\Omega}\varphi\ominus_a\varphi_1(t,y(t))d\mu(t)\leq 1,
$$
which proves the inequality (\ref{claim}).
It means that $y\in L^{\varphi\ominus\varphi_1}$ and 
$$
\norm{y}_{\varphi\ominus\varphi_1}\leq 1.
$$
%what finishes proof of the inequality.
Concluding, if $0\leq y\in M(L^{\varphi_1},L^{\varphi})$ is a simple function,  then $y\in L^{\varphi\ominus\varphi_1}$ and 
$$
\norm{y}_{\varphi\ominus\varphi_1}\leq 4c\norm{y}_M.
$$
Thus the theorem is proved for positive simple functions. We will once again use the Fatou property to complete the argument for an arbitrary function.

Let now $y\in M(L^{\varphi_1},L^{\varphi})$ be arbitrary. There exists a sequence of simple functions $(y_n)$ such that $0\leq y_n\uparrow \abs{y}$ a.e. on $\Omega$. Since $M(L^{\varphi_1},L^{\varphi})$ is a Banach ideal space, $\norm{y_n}_M\leq\norm{y}_M$ for every $n\in\mathbb{N}$. %Since $M(L^{\varphi_1},L^{\varphi})$ has Faotu property sequence  $(y_n)$ can be chosen
 From the Fatou property of $L^{\varphi\ominus\varphi_1}$ we have $y\in L^{\varphi\ominus\varphi_1}$ and
$$
\norm{y}_{\varphi\ominus\varphi_1}\leq\sup\limits_{n\in\mathbb{N}}\norm{y_n}_{\varphi\ominus\varphi_1}
\leq 4c\norm{y_ n}_M\leq 4c\norm{y}_M,
$$
which finishes the proof.
\end{proof}

In the special case of variable exponent spaces we have the following corollary. It has been recently proved in  \cite{Kar18} using elementary methods.
Recall that the variable exponent space (or Nakano  space) is defined as $L^{p(\cdot)}:=L^{\varphi}$, where $\varphi(t,u)=u^{p(t)}$, for a measurable function $p:\Omega\rightarrow [1,\infty)$.

\begin{corollary}
Let  $(\Omega,\Sigma,\mu)$ be non-atomic and let $p,q:\Omega\rightarrow [1,\infty)$ be two measurable functions satisfying $q(t)\leq p(t)$ for $\mu$-a.e. $t\in\Omega$. Then 
$$
M(L^{p(\cdot)},L^{q(\cdot)})=L^{r(\cdot)},
$$
where 
$\frac{1}{p(t)}+\frac{1}{r(t)}=\frac{1}{q(t)}$ for $\mu$-a.e. $t\in\Omega$. 
\end{corollary}
\proof First of all, observe that each Nakano space $L^{p(\cdot)}$ may be equivalently defined by the Musielak--Orlicz function $\varphi_p(t,u)=\frac{1}{p(t)}u^{p(t)}$. In fact,
we see that for $\varphi(t,u)=u^{p(t)}$ there holds
$$
\varphi(t,\frac{u}{2})=(\frac{u}{2})^{p(t)}\leq \frac{1}{p(t)}u^{p(t)}=\varphi_p(t,u)\leq \varphi(t,u),
$$
for each $t\in \Omega$ and $u>0$, which means that  $L^{p(\cdot)}=L^{\varphi_p}$.
Now the proof follows directly from Example \ref{exam1} and the above theorem.
\endproof

\section{Pointwise products}

If $\varphi, \varphi_1, \varphi_2$ are Musielak--Orlicz functions we write
$\varphi_1^{-1}\varphi_2^{-1}\prec\varphi^{-1}$ if there exists a constant $C>0$ such that
$$
C\varphi_1^{-1}(t,u)\varphi_2^{-1}(t,u)\leq\varphi^{-1}(t,u)
$$
for a.e. $t\in\Omega$ and each $u\geq 0$. Similarly, we write
$\varphi_1^{-1}\varphi_2^{-1}\succ\varphi^{-1}$ if there exists a constant $C>0$ such that for a.e. $t\in\Omega$ and each $u\geq 0$
$$
C\varphi_1^{-1}(t,u)\varphi_2^{-1}(t,u)\geq\varphi^{-1}(t,u).
$$
Moreover,  $\varphi_1^{-1}\varphi_2^{-1}\approx\varphi^{-1}$ means that  $\varphi_1^{-1}\varphi_2^{-1}\prec\varphi^{-1}$ and $\varphi_1^{-1}\varphi_2^{-1}\succ\varphi^{-1}$.

%%%%%%%%%%%%%%%%%%%%%%%%%%%%
%{\color{red} NOWE}
Recall the classical Lozanovskii factorization theorem (see \cite[Theorem 6]{Lo69}, cf. \cite{Gi81}) which says that each Banach ideal space $E$ factorizes $L^1$, this is
$$
E\odot M(E,L^1)=L^1.
$$
Generalizing this idea, for a couple of Banach ideal spaces $E,F$ we say that $E$ factorizes $F$ if
$$
E\odot M(E,F)=F,
$$
(see \cite[Section 6]{KLM14} for a discussion of the general factorization problem). 
Recently the authors proved in \cite[Theorem 2]{LT17} that for a pair of Young functions $\varphi, \varphi_1$, the function space $L^{\varphi_1}$ may be factorized by $L^{\varphi}$ if and only if
\begin{equation}
\varphi_1^{-1}(\varphi\ominus\varphi_1)^{-1}\approx\varphi^{-1}.\label{warunek}
\end{equation}
That result is based on Theorem 5 in \cite{KLM14}, which states that in the case of non-atomic and finite measure space, given three Young functions $\varphi,\varphi_0, \varphi_1$, there holds
$$
L^{\varphi}\odot L^{\varphi_1}=L^{\varphi_0}
$$
if and only if
$$
\varphi_1^{-1}\varphi_0^{-1}\approx\varphi^{-1}.
$$

In this section we will show that, in the case of Musielak--Orlicz spaces, the condition (\ref{warunek}) is sufficient, but not necessary to have the factorization
$$
L^{\varphi}\odot L^{\varphi_1}=L^{\varphi_0}.
$$
An immediate consequence of Theorem \ref{th1} is the following inclusion.
%%%%%%%%%%%%%%%%%%%%%%%%%%%%

\begin{lemma}
Let $\varphi, \varphi_1$ be Musielak--Orlicz functions over $(\Omega,\Sigma,\mu)$. If $\supp(L^{\varphi_1})=\Omega$ then
$$
L^{\varphi}\subset L^{\varphi\ominus\varphi_1}\odot L^{\varphi_1} 
$$
\end{lemma}

\begin{proof}
Let $x\in L^{\varphi\ominus\varphi_1}$ and $y\in L^{\varphi_1}$. Then, since $M(L^{\varphi_1},L^{\varphi})=L^{\varphi\ominus\varphi_1}$, we see that
$$
xy\in L^{\varphi}
$$
and
$$
\norm{xy}_{\varphi}\leq\norm{x}_{M}\norm{y}_{\varphi_1}\leq c\norm{x}_{\varphi\ominus\varphi_1}\norm{y}_{\varphi_1}.
$$
\end{proof}

\begin{lemma}
Let  $\varphi,\varphi_0, \varphi_1$ be Musielak--Orlicz functions over $(\Omega,\Sigma,\mu)$. Assume that $\varphi_1^{-1}\varphi_0^{-1}\succ\varphi^{-1}$ and $\supp L^{\varphi_1}=\Omega$. Then
$$
L^{\varphi}\subset L^{\varphi_0}\odot L^{\varphi_1}.
$$
\end{lemma}

\begin{proof}
Denote by $c\geq 1$ the constant of inclusion 
$$
L^{\varphi}[\Omega_{\infty,\infty}]\subset L^{\infty}(b_{\varphi}^{-1})[\Omega_{\infty,\infty}].
$$
Let $0\leq z\in L^{\varphi}$ be such that $\norm{z}_{\varphi}=\frac{2}{3c}$.
Put $y(t):=\varphi(t,z(t))$. We have $y(t)<\infty$ a.e., since $z(t)\leq \frac{2}{3}b_{\varphi}(t)$. For $i=0,1$, define
$$z_i(t) :=
 \begin{cases}
  \varphi_i^{-1}(t,y(t))\sqrt{\frac{z(t)}{\varphi_0^{-1}(t,y(t))\varphi_1^{-1}(t,y(t))}} &\mbox{if } t\in\supp(z)\\
  0 &\mbox{if } t\notin\supp(z) 
  \end{cases} $$ 
  Note that  $z=z_0z_1$. We will show that $z_i\in L^{\varphi_i}$ for $i=0,1$.  Let $D>0$  be such that 
  $$
  D\varphi_1^{-1}(t,u)\varphi_0^{-1}(t,u)\geq\varphi^{-1}(t,u).
  $$
 We claim that
 \begin{equation}\label{zleqy}
 \varphi_i(t,\frac{z_i(t)}{\sqrt{D}})\leq y(t).%\varphi(t,z(t)).
 \end{equation}
  If $y(t)=0$ then
  $$
  z_i(t)=a_{\varphi}(t)\sqrt{\frac{z(t)}{a_{\varphi_0}(t)a_{\varphi_1}(t)}}
  \leq  a_{\varphi}(t)\sqrt{\frac{a_{\varphi}(t)}{a_{\varphi_0}(t)a_{\varphi_1}(t)}}
  \leq  a_{\varphi}(t)\sqrt{D}$$ thus   $$\varphi_i(t,\frac{z_i(t)}{\sqrt{D}})=0.%=y(t).
  $$
  If $y(t)>0$ then
  \begin{align*}
  z_i(t)&=\varphi_i^{-1}(t,y(t))\sqrt{\frac{z(t)}{\varphi_0^{-1}(t,y(t))\varphi_1^{-1}(t,y(t))}} \\
  &\leq \varphi_i^{-1}(t,y(t))\sqrt{\frac{Dz(t)}{\varphi^{-1}(t,y(t))}} \\
  &= \varphi_i^{-1}(t,y(t))\sqrt{\frac{Dz(t)}{z(t)}}=\varphi_i^{-1}(t,y(t))\sqrt{D}.
  \end{align*}
 Therefore, 
  $$
  \varphi_i(t,\frac{z_i(t)}{\sqrt{D}})\leq\varphi_i(t,\varphi_i^{-1}(t,y(t)))=y(t)
  $$
and the claim is proved.
Integrating both sides in (\ref{zleqy}) we obtain 
  $$
  \I{\frac{z_i}{\sqrt{D}}}{\varphi_i}\leq\I{z}{\varphi}\leq 1,
  $$
  for $i=0,1$. It follows, that 
  $$
  \norm{z_i}_{\varphi_i}\leq\sqrt{D}\leq\sqrt{2Dc\norm{z}_{\varphi}}.
  $$
  This means that  $z\in  L^{\varphi_0}\odot L^{\varphi_1}$ and 
  $$
   \norm{z}_{L^{\varphi_0}\odot L^{\varphi_1}}\leq 2Dc\norm{z}_{\varphi}.
   $$
\end{proof}

Recall that for Musielak--Orlicz functions $\varphi, \varphi_1$,  the generalized Young inequality implies that 
$$
\varphi_1^{-1}(\varphi\ominus\varphi_1)^{-1}\prec\varphi^{-1}
$$ 
(see for example \cite{KLM13}).

\begin{corollary}
Let $\varphi, \varphi_1$ be Musielak--Orlicz functions over a  measure space $(\Omega,\Sigma,\mu)$. If $\varphi_1^{-1}(\varphi\ominus\varphi_1)^{-1}\approx\varphi^{-1}$
then $L^{\varphi_1}$ factorizes $L^{\varphi}$.
\end{corollary}

We finish the paper providing an example, which shows that the opposite implication does not hold. In particular, Theorem 2 in \cite{LT17} cannot be directly generalized to Musielak--Orlicz spaces.

\begin{example}\label{przyklad}
Consider $\Omega=[0,1/2)$ with the Lebesgue measure. Define the following Musielak--Orlicz functions 
$$
\varphi(t,u)=\max\{u-t,0\},
$$ 
$$
\varphi_1(t,u)=u,
$$
for $t\in \Omega$ and $u\geq 0$.
Then  $L^{\varphi}=L^{\varphi_1}=L^1$. Moreover, we have  
  $$
  \varphi\ominus\varphi_1(t,u)=
  \begin{cases}
  0 &\mbox{if } 0\leq u\leq 1\\
  \infty &\mbox{if } u>1,
  \end{cases} 
  $$
  thus $L^{\varphi\ominus\varphi_1}=L^{\infty}$. In consequence, the  factorization 
   $$
   L^{\varphi_1}\odot L^{\varphi\ominus\varphi_1}=L^{\varphi}
   $$
   holds.
On the other hand an easy computations show that 
$$
(\varphi\ominus\varphi_1)^{-1}(t,u)=1,\ \varphi^{-1}(t,u)=u+t,\ \varphi_1^{-1}(t,u)=u.
$$
  We have $\varphi_1^{-1}(t,u)(\varphi\ominus\varphi_1)^{-1}(t,u)=u$, thus there is no constant $D$ such that
   $$
   D\varphi_1^{-1}(t,u)(\varphi\ominus\varphi_1)^{-1}(t,u)\geq\varphi^{-1}(t,u)
   $$
   for every $t\in \Omega$ and $u\geq 0$ (take for example $u=0$ and $t>0$).  Hence
   $$
   \varphi_1^{-1}(\varphi\ominus\varphi_1)^{-1}\nsucc\varphi^{-1}.
   $$

\end{example}

%This example shows that $\varphi_1^{-1}(\varphi\ominus\varphi_1)^{-1}\varphi^{-1}$ is not necessary condition for factorization $$L^{\varphi}=L^{\varphi_1}\odot M(L^{\varphi_1},L^{\varphi}),$$ thus for Musielak--Orlicz spaces 

%\begin{theorem}
%Let $\varphi, \varphi_1$ be Musielak--Orlicz functions. %Assume that $\supp(L^{\varphi})\subset\supp(L^{\varphi_1})$ then
%Assume that $\supp(L^{\varphi_1})=\Omega$. Then  $L^{\varphi_1}$ factorizes $L^{\varphi}$ i.e.
%$$L^{\varphi}=L^{\varphi_1}\odot M(L^{\varphi_1},L^{\varphi})$$
%if $\varphi_1^{-1}(\varphi\ominus\varphi_1)^{-1}\succ\varphi^{-1}$.
%\end{theorem}
%\begin{proof}
%\end{proof}

%\begin{question}
%For three Musielak--Orlicz functions $\varphi, \varphi_1, \varphi_2$ define a relation 
%$\varphi_1^{-1}\varphi_2^{-1}\sqsupset\varphi^{-1}$ if there exist a integrable function $C:\Omega\rightarrow [0,\infty)$ such that for a.e. $(t,u)\in\Omega\times[0,\infty)$
%$$C(t)\varphi_1^{-1}(t,u)\varphi_2^{-1}(t,u)\geq\varphi^{-1}(t,u).$$ Is $\varphi_1^{-1}(\varphi\ominus\varphi_1)^{-1}\sqsupset\varphi^{-1}$ equivalent to fact that 
%$$L^{\varphi}=L^{\varphi_1}\odot M(L^{\varphi_1},L^{\varphi})?$$
%

%\end{question}
\section{Acknowledgements}
The authors have been partially supported by the Grant 04/43/DSPB/0106 from the Polish Ministry of
Science and Higher Education

\end{document}